\documentclass[reqno]{amsart}
\usepackage{amsmath} 

\usepackage{amsthm}
\usepackage{amssymb}
\usepackage{graphics}
\usepackage{latexsym}

\numberwithin{equation}{section}
\newtheorem{thm}{Theorem}[section]

\newtheorem{lem}[thm]{Lemma}

\newcommand{\LR}[1]{{\langle {#1} \rangle }}

\newcommand{\EQ}[1]{\begin{equation} \begin{split} #1
 \end{split} \end{equation}}
\newcommand{\EQS}[1]{\begin{align} #1 \end{align}}

\title[Ill-posedness for HWS equation]{Ill-posedness for the Half wave Schr\"odinger  equation} 

\author[I. Kato]{Isao Kato} %\\ 
\email[Isao Kato]{kato.isao.23n@st.kyoto-u.ac.jp}


\begin{document}

\begin{abstract}
We study ill-posedness for the half wave Schr\"odinger equation introduced by Xu \cite{Xu}. 
Ill-posedness is obtained in the super-critical or at the critical space. 
The proof is based on the argument established by Christ, Colliander and Tao \cite{CCT1}. 
For the critical space, we use the standing wave solution, which was proved the existence by Bahri, Ibrahim and Kikuchi \cite{BIK}.     
\end{abstract}
\maketitle
\setcounter{page}{001}

%%%%%%%%%%%%%%%%%%%%%%%%%%%%%%%%%%%%%%%%%%%%%%%%%%%%%%%%%%%%%%%%%%%%%%%%%%%%%
%%%%%%%%%%%%%%%%%%%%%%%%%%%%%%%%%  Section 1  %%%%%%%%%%%%%%%%%%%%%%%%%%%%%%%
%%%%%%%%%%%%%%%%%%%%%%%%%%%%%%%%%%%%%%%%%%%%%%%%%%%%%%%%%%%%%%%%%%%%%%%%%%%%%
\section{Introduction} 
We consider the Cauchy problem for the following equation: 
\EQS{ \label{NLS} 
 \begin{cases} 
 i\partial_t u + \partial_x^2 u - |D_y| u = \mu |u|^{p-1}u, \qquad (t, x, y) \in \mathbb{R}^3, \\ 
 u(0, x, y) = u_0(x, y),  
 \end{cases} 
}
where $|D_y| = (-\partial_y^2)^{\frac{1}{2}}, \mu = \pm 1, p > 1$. 
\eqref{NLS} with $\mu = 1, p = 3$ is firstly considered by Xu \cite{Xu} on the cylinder $(x, y) \in \mathbb{R} \times \mathbb{T}$.    
In \cite{Xu}, large time behavior for \eqref{NLS}, namely the modified scattering for small regular solution is derived. 
For the case $\mu = -1$ and $1 < p < 5$, Bahri, Ibrahim and Kikuchi \cite{BIK} proved existence and the stability of the ground states $Q_{\omega}(x, y)$ with $\omega > 0$ which satisfy 
\begin{align*}
 -\partial_x^2 Q + |D_y|Q + \omega Q - |Q|^{p-1}Q = 0. 
\end{align*} 
They also found the traveling wave solutions $e^{i\omega t} Q_{\omega, v}(x, y-vt), \omega > 0, v \in \mathbb{R}$ exist and derived the behavior of $Q_{\omega, v}$ in $H^{1, 0}(\mathbb{R}^2)$ or $L^2(\mathbb{R}^2)$ as the velocity $v$ tends to $1$. 
See the definition of $H^{s_1, s_2}(\mathbb{R}^2)$ in the following section.   
Here $Q_{\omega, v}$ is a solution to the following equation: 
\begin{align*}
 -\partial_x^2 Q + |D_y|Q + iv\partial_y Q + \omega Q - |Q|^{p-1} Q = 0. 
\end{align*} 

Concerning the Cauchy problem, \cite{BIK} showed the local well-posedness in the scale  sub-critical space, namely $H^{0, s}(\mathbb{R}^2)$ with $s > \frac{1}{2}$ and $1 < p \leqslant 5$ by the standard fixed point argument (see section 2 for the definition of scale sub-critical space)\footnote{We can also obtain the local well-posedness in $H^{1, 0}(\mathbb{R}^2) \cap H^{0, s}(\mathbb{R}^2), s > \frac{1}{2}$ for \eqref{NLS}, $\mu = \pm 1$.}.   
The following Strichartz estimate is a key role to show local well-posedness: 
For $S(t) := e^{it(\partial_x^2 - |D_y|)}$ and $s > \frac{1}{2}$, there exists $C > 0$ such that  
\begin{align*} 
 &\|S(t)f\|_{L^4_T L^{\infty}_{x, y}} \leqslant C\|f\|_{H^{0, s}}, \qquad f \in H^{0, s}(\mathbb{R}^2), \\ 
 &\left\| \int_0^t S(t-s)F(s) \, \mathrm{d}s \right\|_{L^4_T L^{\infty}_{x, y}} \leqslant C\|F\|_{L^1_T H^{0, s}},\qquad F \in L^1_T H^{0, s}(\mathbb{R}^2). 
\end{align*} 
So far there is no result for lower regularity space, for instance the energy space $E := H^{1, 0}(\mathbb{R}^2) \cap H^{0, \frac{1}{2}}(\mathbb{R}^2)$ 
equipped with the norm 
\begin{align*}
 \|u\|_E := \left(\|\partial_x u\|_{L^2}^2 + \||D_y|^{\frac{1}{2}} u\|_{L^2}^2 + \|u\|_{L^2}^2  \right)^{\frac{1}{2}}.  
\end{align*} 
In Theorem \ref{norm_inflation} below, we find that for $p > 5$ the ill-posedness (norm inflation) holds in $E$ since $E$ is the super-critical space in this case. 
However for $1 < p < 5$ (resp. $p = 5$), $E$ is the sub-critical (resp. scale-critical) space and  
we do not obtain any result yet.      
We explain some difficulties to show well-posedness in $E$ for $1 < p \leqslant 5$. 
The first one is the restriction of regularity with respect to $y$ variable, namely $s > \frac{1}{2}$ for the above Strichartz estimate. 
We do not know whether the Strichartz estimate hold or not for $s = \frac{1}{2}$. 
If we neglect the term $\partial_x^2 u$ and $x$ variable, \eqref{NLS} is nothing but the half wave equation and it is ill-posed in $H^s(\mathbb{R})$ for $s < \frac{1}{2}$ and $p = 3$  
(see e.g. Theorem 1.3 \cite{CP}). 
For the delicate case $s = \frac{1}{2}$, one may apply the argument in the cubic Szeg\H{o} equation \cite{GG} or the half wave equation \cite{KLR}.  
However, for instance \eqref{NLS} with $p = 3$, the Judovi$\rm{\check{c}}$ argument, which was applied to show uniqueness of the solution for Szeg\H{o} equation \cite{GG} or the half wave equation \cite{KLR}, is difficult to apply.  
In fact, we cannot replace the right-hand side of the following inequality $X$ by the energy space 
$E$. 
%$H^{1, 0}(\mathbb{R}^2) \cap H^{0, \frac{1}{2}}(\mathbb{R}^2)$. 
For $q \in [2, \infty), u \in X = H^{\frac{1}{2}, \frac{1}{2}}(\mathbb{R}^2)\ \text{or} \ H^1(\mathbb{R}^2)$, there exists $C > 0$ such that the following holds:  
\begin{align*}
 \|u\|_{L^q(\mathbb{R}^2)} \leqslant C \sqrt{q}\, \|u\|_{X}. 
\end{align*}
This is the second difficulty.  
Moreover, the Gagliardo-Nirenberg inequality (Lemma 2.1 \cite{BIK}):  
\begin{align*}
 \|u\|_{L^{q+1}(\mathbb{R}^2)}^{q+1} \leqslant C_{\text{GN}} \|\partial_x u\|_{L^2(\mathbb{R}^2)}^{\frac{q-1}{2}} \| |D_y|^{\frac{1}{2}} u\|_{L^2(\mathbb{R}^2)}^{q-1} \|u\|_{L^2(\mathbb{R}^2)}^{\frac{5-q}{2}},\qquad C_{\text{GN}} > 0   
\end{align*} 
holds only for $1 < q < 5$.  
However, we need $q \to \infty$ to show the uniqueness. 
These are the main reasons why we cannot apply the standard technique. 
Thus the Cauchy problem \eqref{NLS} in $E$ for $1 < p \leqslant 5$ is an interesting open problem. 

The aim of this paper is to investigate the ill-posedness for \eqref{NLS} in the super-critical (resp. at the critical) space $H^{s_1, s_2}$ for $\mu = \pm 1$ (resp. $\mu = -1$) since there is no ill-posedness result.   
The first statement is the norm inflation in the super-critical space.  
%Actually this type of strong ill-posedness result is nontrivial.   
\begin{thm} \label{norm_inflation}
 Let $p > 1$ be an odd integer and $\mu = \pm 1$. 
 Suppose that either 
 $s_1, s_2 \geqslant 0$ except for $s_1 = s_2 = 0$ and $s_1 + 2s_2 < \frac{3}{2}-\frac{2}{p-1}$ 
 or 
 $s_1, s_2 \leqslant -\frac{1}{2}$. 
 Then for any $\varepsilon > 0$ there exist a solution $u$ of \eqref{NLS} and $t > 0$ such that $u(0) \in \mathcal{S}(\mathbb{R}^2)$, 
 \begin{align*}
  \|u(0)\|_{H^{s_1, s_2}} < \varepsilon, \qquad 
  0 < t < \varepsilon, \qquad  
  \|u(t)\|_{H^{s_1, s_2}} > \frac{1}{\varepsilon}.  
 \end{align*} 
%In particular, for any $t > 0$, the solution map $\mathcal{S}(\mathbb{R}^2) \ni u(0) \mapsto u(t)$ for the Cauchy problem \eqref{NLS} fails to be continuous at $0$ in $H^{s_1, s_2}$. 
 If $p > 1$ is not an odd integer, then the same conclusion holds provided that there exists an integer $k > 1$ such that $p \geqslant k+1$ and either 
 $s_1, s_2 \geqslant 0$ except for $s_1 = s_2 = 0$ and $s_1 + 2s_2 < \frac{3}{2} - \frac{2}{p-1}$ 
 or 
 $s_1, s_2 \leqslant -\frac{1}{2}$.  
\end{thm}
The second result is the decoherence in $H^{s_1, s_2}$ with non-positive $s_1, s_2$. 
Here, the decoherence means that the flow map is not uniformly continuous at time $0$ in $H^{s_1, s_2}$. 
\begin{thm} \label{Decoherence}
 Let $p > 1$ be an odd integer and $\mu = \pm 1$. 
 For any $s_1, s_2 \leqslant 0, 0 < \delta, \varepsilon < 1$ and $t > 0$ 
 there exist $C > 0$ and solutions of \eqref{NLS} with initial data $u_1(0), u_2(0) \in \mathcal{S}(\mathbb{R}^2)$ such that 
 \begin{align}
  &\|u_1(0)\|_{H^{s_1, s_2}}, \|u_2(0)\|_{H^{s_1, s_2}} \leqslant C\varepsilon,  \label{estimate_initialdata} \\ 
  &\|u_1(0) - u_2(0)\|_{H^{s_1, s_2}} \leqslant C\delta, \label{difference_initialdata} \\ 
  &\|u_1(t) - u_2(t)\|_{H^{s_1, s_2}} \geqslant C\varepsilon. \label{difference_solution}   
 \end{align}  
If $p > 1$ is not an odd integer, then the same conclusion holds provided that there exists an integer $k > \frac{3}{2}$ such that $p \geqslant k + 1$ and $s_1, s_2 \leqslant 0$.  
\end{thm}
The proof of Theorems \ref{norm_inflation} and \ref{Decoherence} are based on the work of Christ, Colliander and Tao \cite{CCT1} (Theorems 2, 1). 
In \cite{CCT1}, they proved the norm inflation and the decoherence for the Schr\"odinger equation and the wave equation with power type nonlinearity $\mu |u|^{p-1}u, \mu = \pm 1$, by regarding the nonlinearity as the main term and the linear dispersive part as the perturbation.   
To show the decoherence for the Schr\"odinger equation, the Galilean invariance parameter plays a crucial role. 
Using this idea with suitable modification (since \eqref{NLS} has different dispersion in $x$ and $y$ variable), we can obtain the ill-posedness for \eqref{NLS}. 
We remark that we consider the ill-posedness in the super-critical space in Theorems \ref{norm_inflation} and \ref{Decoherence} and in Theorem \ref{Decoherence} more regularity (or weight) $k$ is needed than in Theorem 1 \cite{CCT1}.    

The third main result is also the decoherence, and we consider at the scale critical space with non-negative $s_1, s_2$.  
\begin{thm} \label{decoherence_critical} 
 Let $1 < p < 5$ and $\mu = -1$. 
 For any $s_1, s_2 \geqslant 0, s_1 + 2s_2 = \frac{3}{2} - \frac{2}{p-1}, 0 < \delta, \varepsilon < 1$  and $t > 0$ 
 there exist solutions of \eqref{NLS} with initial data 
 $u_1(0), u_2(0) \in H^{s_1, s_2}(\mathbb{R}^2)$ such that 
 \eqref{difference_initialdata} and \eqref{difference_solution} hold in Theorem \ref{Decoherence}.   
\end{thm} 
We note that in Theorem \ref{decoherence_critical}, we only consider the focusing nonlinearity $\mu = -1$ and $1 < p < 5$ since we use the standing wave solutions for \eqref{NLS} with $\mu = -1$ obtained by \cite{BIK}. 

Let us introduce an organization of this paper. 
In section 2, we give some notations and lemmas (Lemmas \ref{lemma1}, \ref{lemma2}). 
By the usual energy method, we obtain Lemma \ref{lemma1}, namely local well-posedness for  \eqref{small_dispersion} (the equation with small dispersion parameter $\nu$) with regular initial data.  
In section 3, we prove the norm inflation in the super-critical space. 
In section 4, we prove the decoherence. 
We apply Lemma \ref{lemma2} together with the scaling to show the decoherence occurs in a very short time.  
Finally in section 5, we prove the decoherence at the critical space for the focusing case.

%%%%%%%%%%%%%%%%%%%%%%%%%%%%%%%%%%%%%%%%%%%%%%%%%%%%%%%%%%%%%%%%%%%
%%%%%%%%%%%%%%%%%%%%%%%%%% Section 2 %%%%%%%%%%%%%%%%%%%%%%%%%%%%%%
%%%%%%%%%%%%%%%%%%%%%%%%%%%%%%%%%%%%%%%%%%%%%%%%%%%%%%%%%%%%%%%%%%%
\section{Preliminaries} 
$\hat{f}$ denotes the Fourier transform of $f$ with respect to the spatial variable $x$ and $y$. 
We denote the anisotropic Sobolev spaces $H^{s_1, s_2}(\mathbb{R}^2)$ and $\dot{H}^{s_1, s_2}$ as  
\begin{align*}
 &H^{s_1, s_2}(\mathbb{R}^2) = \{ f \in \mathcal{S}'(\mathbb{R}^2)\, | \, \|f\|_{H^{s_1, s_2}} < \infty \}, \\ 
 &\|f\|_{H^{s_1, s_2}} := \left(\int_{\mathbb{R}^2} \LR{\xi}^{2s_1} \LR{\eta}^{2s_2} |\hat{f}(\xi, \eta)|^2 \, \mathrm{d}\xi \, \mathrm{d}\eta \right)^{\frac{1}{2}}, \qquad \LR{\cdot} := \sqrt{1 + |\cdot|^2}, \\ 
 &\dot{H}^{s_1, s_2}(\mathbb{R}^2) = \{ f \in \mathcal{S}'(\mathbb{R}^2)\, | \, \|f\|_{\dot{H}^{s_1, s_2}} < \infty \}, \\ 
 &\|f\|_{\dot{H}^{s_1, s_2}} := \left(\int_{\mathbb{R}^2} |\xi|^{2s_1} |\eta|^{2s_2} |\hat{f}(\xi, \eta)|^2 \, \mathrm{d}\xi \, \mathrm{d}\eta \right)^{\frac{1}{2}}. 
\end{align*} 

\eqref{NLS} is scale-invariant under the transformation 
\begin{align*}
 u_{\lambda}(t, x, y) = \lambda^{-\frac{2}{p-1}} u\left( \frac{t}{\lambda^2}, \frac{x}{\lambda}, \frac{y}{\lambda^2} \right)   
\end{align*} 
for all $\lambda > 0$. 
If $s_1, s_2$ satisfy 
\begin{align} 
 s_1 + 2s_2 = \frac{3}{2} - \frac{2}{p-1}, 
                         \label{scale_critical}
\end{align}  
then we have 
$\|u_{\lambda}(0)\|_{\dot{H}^{s_1, s_2}} = \|u_0\|_{\dot{H}^{s_1, s_2}}$. 
We call $H^{s_1, s_2}$ scale critical\footnote{The energy space $H^{1, 0}(\mathbb{R}^2) \cap H^{0, \frac{1}{2}}(\mathbb{R}^2)$ is scale critical when $p = 5$.} if $s_1, s_2$ satisfy \eqref{scale_critical} and super-critical if $s_1, s_2$ satisfy 
\begin{align*}
 s_1 + 2s_2 < \frac{3}{2} - \frac{2}{p-1}. 
\end{align*} 
Similarly the scale sub-critical space $H^{s_1, s_2}$ be such that $s_1, s_2$ satisfy 
\begin{align*}
 s_1 + 2s_2 > \frac{3}{2} - \frac{2}{p-1}. 
\end{align*}

$\mathcal{H}^{k, k}$ denotes the weighted Sobolev space endowed with the norm 
\begin{align*} 
 \|u\|_{\mathcal{H}^{k, k}} := \sum_{i = 0}^k \| (1 + |x| + |y|)^{k-i}\, \partial^i u\|_{L^2(\mathbb{R}^2)}. 
\end{align*}

Let $\phi_0 = aw,  a \in [\frac{1}{2}, 1]$ and the Schwartz function $w \in \mathcal{S}(\mathbb{R}^2)$ be given and let $\phi = \phi^{(a, \nu)}(t, x, y)$ satisfies  
\EQS{ \label{small_dispersion}
 \begin{cases} 
  i\partial_t \phi + \nu^2 \partial_x^2 \phi - \nu^2 |D_y| \phi = \mu |\phi|^{p-1} \phi, \\ 
  \phi(0, x, y) = \phi_0(x, y) \in \mathcal{S}(\mathbb{R}^2), 
 \end{cases} 
} 
where $0 < \nu \ll 1, \mu = \pm 1$.  
By letting $\nu \to 0$, \eqref{small_dispersion} approaches 
\EQS{ \label{ode} 
 \begin{cases} 
  i\partial_t \phi = \mu |\phi|^{p-1} \phi, \\ 
  \phi(0, x, y) = \phi_0(x, y) \in \mathcal{S}(\mathbb{R}^2).    
 \end{cases} 
}
The solution $\phi = \phi^{(0)}$ of \eqref{ode} can be expressed as 
\begin{align} 
 \phi^{(0)} = \phi_0 e^{-i\mu t |\phi_0|^{p-1}}. 
                     \label{phi0_ode}
\end{align}

For $\frac{1}{2} \leqslant a \leqslant 1, 0 < \lambda \ll 1, v \in \mathbb{R}$ we set 
\begin{align}
 u^{(a, \nu, \lambda, v)}(t, x, y) := \lambda^{-\frac{2}{p-1}}e^{-\frac{i}{2}v \cdot x} e^{-\frac{i}{4}|v|^2 t} \phi^{(a, \nu)}\left( \frac{t}{\lambda^2}, \frac{\nu}{\lambda}(x+vt), \frac{\nu^2}{\lambda^2}y \right). 
                       \label{u}
\end{align}
Then $u = u^{(a, \nu, \lambda, v)}$ satisfies 
\begin{align*}
 i\partial_t u + \partial_x^2 u - |D_y| u = \mu |u|^{p-1} u. 
\end{align*}

The following lemma will be applied in the proof of Theorems \ref{norm_inflation}, \ref{Decoherence}. 
\begin{lem} {\rm(Lemma 2.1  \cite{CCT1})} \label{lemma1}
Let $p \geqslant 1, k > 1$ and let $\phi^{(0)}(t)$ satisfies \eqref{phi0_ode}. 
If $p$ is not an odd integer, then we also assume $p \geqslant k + 1$. 
Then there exist $C, c > 0$ depending on $p, k$ such that the following holds: 
For sufficiently small $0 < \nu \leqslant c$ and $T = c|\log \nu|^c$, 
there exists a solution $\phi(t) \in C^1([-T, T]; \mathcal{H}^{k, k})$ of \eqref{small_dispersion} satisfying    
\begin{align} 
 \|\phi(t) - \phi^{(0)}(t)\|_{\mathcal{H}^{k, k}} \leqslant C\nu
                  \label{est_small_dispersion}
\end{align} 
for all $|t| \leqslant c|\log \nu|^c$.   
\end{lem}

\begin{proof}
Let $\phi = \phi^{(0)} + \psi$ where $\phi, \phi^{(0)}$ satisfy \eqref{small_dispersion}, \eqref{phi0_ode} respectively.  
Then $\psi$ satisfies 
\begin{align*} 
 \begin{cases} 
  i\partial_t \psi + \nu^2 \partial_x^2 \psi - \nu^2 |D_y| \psi = -\nu^2 \partial_x^2 \phi^{(0)} + \nu^2 |D_y| \phi^{(0)} + \mu |\phi^{(0)} + \psi|^{p-1} (\phi^{(0)} + \psi) - \mu |\phi^{(0)}|^{p-1} \phi^{(0)},  \\ 
  \psi(0, x, y) = 0.   
  \end{cases} 
\end{align*}  
To show \eqref{est_small_dispersion} we need to prove  
\begin{align} \label{apriori_bound}
 \sup_{|t| \leqslant T} \|\psi(t)\|_{\mathcal{H}^{k, k}} \leqslant C \nu   
\end{align}      
for $0 \leqslant T \leqslant c|\log \nu|^c$.  
For simplicity, let $t \geqslant 0$. 
The energy inequality gives 
\begin{align}
 \partial_t \|\psi(t)\|_{\mathcal{H}^{k, k}} 
 &\leqslant C\| -\nu^2(\partial_x^2 - |D_y|)\phi^{(0)}(t) + \mu \bigl(|\phi^{(0)} + \psi|^{p-1}(\phi^{(0)} + \psi) - |\phi^{(0)}|^{p-1} \phi^{(0)}\bigr)(t) \|_{\mathcal{H}^{k, k}} \notag \\ 
 &\hspace{3em} + C \|\psi(t)\|_{\mathcal{H}^{k, k}}. 
                   \label{energy_ineq}
\end{align} 
From $\eqref{phi0_ode}, \phi_0 \in \mathcal{S}(\mathbb{R}^2)$ and 
\begin{align} 
 \|\phi^{(0)}(t)\|_{\mathcal{H}^{k, k}} + \|\phi^{(0)}(t)\|_{C^k} \leqslant C(1+t)^k, 
                   \label{decay_initial}
\end{align}
we obtain 
\begin{align} 
 \|(\partial_x^2 - |D_y|)\phi^{(0)}(t)\|_{\mathcal{H}^{k, k}} \leqslant C(1+t)^k. 
              \label{phi0_linear_est} 
\end{align} 
We show 
\begin{align} 
 \|F(\phi^{(0)}+\psi)(t) - F(\phi^{(0)})(t)\|_{\mathcal{H}^{k, k}} 
 \leqslant C(1+t)^C \|\psi(t)\|_{\mathcal{H}^{k, k}}(1+\|\psi(t)\|_{\mathcal{H}^{k, k}})^{p-1},   
              \label{inhomo_est} 
\end{align}
where $F: \mathbb{C} \to \mathbb{C}$, $F(z) := \mu |z|^{p-1} z$.   
For $j = 0, 1, ... , k$, we see 
\begin{align*}
 &|F^{(j)}(z)| \leqslant C(1+|z|)^{p-j}, \\ 
 &|F^{(j)}(z+\tilde{z}) - F^{(j)}(z)| \leqslant C |\tilde{z}| \, (1+|z|+|\tilde{z}|)^{p-j-1}.  
\end{align*} 
Thus by the chain rule, the Leibnitz rule and \eqref{decay_initial} implies 
\begin{align*}
 &|F^{(j)}(\phi^{(0)}+\psi)(t, x, y) - F^{(j)}(\phi^{(0)})(t, x, y)| \\ 
 &\leqslant C(1+t)^C \sum_{\alpha_1, ... , \alpha_r}|\psi^{(\alpha_1)}(t, x, y)| \cdots |\psi^{(\alpha_r)}(t, x, y)| (1+|\psi(t, x, y)|^{p-r})   
\end{align*} 
for each $0 \leqslant j \leqslant k$, where $\alpha_1, ... , \alpha_r$ range over all finite collections of non-negative integers with $1 \leqslant r \leqslant k+1$ and $\alpha_1 + ... + \alpha_r \leqslant k$. 
Thus \eqref{inhomo_est} follows from the H\"older inequality and the Sobolev inequality. 
Collecting \eqref{energy_ineq}, \eqref{phi0_linear_est} and \eqref{inhomo_est}, 
\begin{align*}
 \partial_t \|\psi(t)\|_{\mathcal{H}^{k, k}} \leqslant C\nu^2(1+t)^C + C(1+t)^C \|\psi(t)\|_{\mathcal{H}^{k, k}} + C(1+t)^C \|\psi(t)\|_{\mathcal{H}^{k, k}}^p,   
\end{align*} 
where $C > 0$ depends on $k$. 
Under the a priori assumption $\|\psi(t)\|_{\mathcal{H}^{k, k}} \leqslant 1$, the above inequality becomes 
\begin{align*}
 \partial_t \|\psi(t)\|_{\mathcal{H}^{k, k}} \leqslant C\nu^2 (1+t)^C + C(1+t)^C \|\psi(t)\|_{\mathcal{H}^{k, k}}^p.  
\end{align*}
By the Gronwall inequality and $\psi(0) = 0$, we have 
\begin{align*}
 \|\psi(t)\|_{\mathcal{H}^{k, k}} \leqslant C \nu^2 e^{C(1+t)^C}, 
\end{align*}
when $0 \leqslant t \leqslant c|\log \nu|^c$ for suitably chosen $c$ and sufficiently small $\nu > 0$. 
Then \eqref{apriori_bound} follows and the a priori assumption $\|\psi(t)\|_{\mathcal{H}^{k, k}} \leqslant 1$ is attained. 
Hence we obtain the desired result.    
\end{proof} 

We prepare the following lemma for the proof of Theorem \ref{Decoherence}. 
\begin{lem} {\rm(Lemma 3.1 \cite{CCT1})} \label{lemma2} 
Let $0 \neq w \in \mathcal{S}(\mathbb{R}^2)$ and $s_1, s_2 < 0$. 
Suppose that $p, k, \phi = \phi^{(a, \nu)}, \phi^{(0)}$ be given in Lemma \ref{lemma1}.  
Also assume that $|v| \geqslant 1, \frac{1}{2} \leqslant a, a' \leqslant 2$ and $0 < \nu \leqslant \lambda \ll 1$. 
Then for $u^{(a, \nu, \lambda, v)}$ given by \eqref{u} there exist $C, c > 0$ such that the followings hold: 
\begin{align} 
 &\|u^{(a, \nu, \lambda, v)}(0)\|_{H^{s_1, s_2}} \leqslant C \lambda^{-\frac{2}{p-1}} |v|^{s_1}\left( \frac{\lambda}{\nu} \right)^{\frac{3}{2}}, \label{est_initialdata} \\ 
 &\|u^{(a, \nu, \lambda, v)}(0) - u^{(a', \nu, \lambda, v)}(0)\|_{H^{s_1, s_2}} \leqslant C \lambda^{-\frac{2}{p-1}}|v|^{s_1} \left( \frac{\lambda}{\nu} \right)^{\frac{3}{2}} |a - a'|,  
                          \label{est_difference_initialdata} \\ 
 &\|u^{(a, \nu, \lambda, v)}(t) - u^{(a', \nu, \lambda, v)}(t)\|_{H^{s_1, s_2}} \notag \\ 
 &\geqslant C \lambda^{-\frac{2}{p-1}} |v|^{s_1} \left( \frac{\lambda}{\nu} \right)^{\frac{3}{2}} \left( \left\| \phi^{(a, \nu)}\left( \frac{t}{\lambda^2} \right) - \phi^{(a', \nu)}\left( \frac{t}{\lambda^2} \right) \right\|_{H^{0, s_2}} - C|\log \nu|^C \left( \frac{\lambda}{\nu} \right)^{-k} |v|^{-s_1-k} \right) 
                          \label{solution_lower_bound}
\end{align}  
whenever $|t| \leqslant c|\log \nu|^c \lambda^2$. 
\end{lem}

\begin{proof}
From \eqref{u}, we have 
\begin{align*}
 [u^{(a, \nu, \lambda, v)}(0)]^{\wedge}(\xi, \eta) 
 = \lambda^{-\frac{2}{p-1}} \left( \frac{\lambda}{\nu} \right)^3 [\phi^{(a, \nu)}(0)]^{\wedge}\left( \frac{\lambda}{\nu}\left( \xi + \frac{v}{2} \right), \frac{\lambda^2}{\nu^2} \eta \right).    
\end{align*} 
By Lemma \ref{lemma1}, for $0 < \nu \leqslant c$ 
there exists the solution $\phi(t, x, y) = \phi^{(a, \nu)}(t, x, y)$ to \eqref{small_dispersion} 
with initial data $\phi^{(a, \nu)}(0, x, y) = aw(x, y),\, 0 \neq w \in \mathcal{S}(\mathbb{R}^2)$.   
From $\frac{1}{2} \leqslant a \leqslant 2$ and $s_1, s_2 < 0$ we have  
\begin{align*} 
 &\|u^{(a, \nu, \lambda, v)}(0)\|_{H^{s_1, s_2}}^2 
 =C a^2 \lambda^{-\frac{4}{p-1}}\left( \frac{\lambda}{\nu} \right)^3 \int_{\mathbb{R}^2} \left(1 + \left|\frac{\nu}{\lambda}\xi - \frac{v}{2}\right| \right)^{2s_1} \left( 1 + \left| \frac{\nu^2}{\lambda^2} \eta \right| \right)^{2s_2} |\widehat{w}(\xi, \eta)|^2 \, \mathrm{d}\xi \, \mathrm{d}\eta \\ 
 &\leqslant C\lambda^{-\frac{4}{p-1}}\left( \frac{\lambda}{\nu} \right)^3 \left( \int_{|\xi| \leqslant \frac{\lambda}{4\nu}|v|} |v|^{2s_1} \| \mathcal{F}_x[w](\xi, \cdot)\|_{L^2_y}^2 \, \mathrm{d}\xi + \int_{|\xi| \geqslant \frac{\lambda}{4\nu}|v|} \| \mathcal{F}_x[w](\xi, \cdot)\|_{L^2_y}\, \mathrm{d}\xi \right) \\ 
 &\leqslant C\lambda^{-\frac{4}{p-1}}\left( \frac{\lambda}{\nu} \right)^3 \left( |v|^{2s_1} + C_N \left( \frac{\lambda}{\nu}|v| \right)^{-N} \right) 
 \leqslant C\lambda^{-\frac{4}{p-1}}\left( \frac{\lambda}{\nu} \right)^3 |v|^{2s_1}
\end{align*}
for all $N < \infty$ since $w \in \mathcal{S}(\mathbb{R}^2)$. 
In the last step we use $\nu \leqslant \lambda$. 
Hence \eqref{est_initialdata} follows. 
%\begin{align} \label{initial_norm}
%\|u^{(a, \nu, \lambda, v)}(0)\|_{H^{s_1, s_2}} \leqslant C\lambda^{-\frac{2}{p-1}}\left( \frac{\lambda}{\nu} \right)^{\frac{3}{2}} |v|^{s_1}. 
%\end{align} 
%%%%%
%$\|u^{(a, \nu, \lambda, v)}(0) - u^{(a', \nu, \lambda, v)}(0)\|_{H^{s_1, s_2}}$. 
\eqref{est_difference_initialdata} is clear from $u^{(a, \nu, \lambda, v)}(0) - u^{(a', \nu, \lambda, v)}(0) = (a - a')u^{(1, \nu, \lambda, v)}(0)$ and \eqref{est_initialdata}. 
%\begin{align} 
% \|u^{(a, \nu, \lambda, v)}(0) - u^{(a', \nu, \lambda, v)}(0)\|_{H^{s_1, s_2}} 
 %&= |a - a'| \|u^{(1, \nu, \lambda, v)}(0)\|_{H^{s_1, s_2}} \notag \\ 
% \leqslant C\lambda^{-\frac{2}{p-1}} \left( \frac{\lambda}{\nu} \right)^{\frac{3}{2}}|v|^{s_1}|a - a'| 
%      \label{initial_difference}
%\end{align}
Finally we prove \eqref{solution_lower_bound} below. 
%estimate $\|u^{(a, \nu, \lambda, v)}(t) - u^{(a', \nu, \lambda, v)}(t)\|_{H^{s_1, s_2}}$. 
From \eqref{u},  
\begin{align*}
 &\|u^{(a, \nu, \lambda, v)}(t) - u^{(a', \nu, \lambda, v)}(t)\|_{H^{s_1, s_2}}^2 \\ 
 &= \lambda^{-\frac{4}{p-1}} \left\|e^{-\frac{i}{2}v \cdot x} \left( \phi^{(a, \nu)}\left( \frac{t}{\lambda^2}, \frac{\nu}{\lambda}(x + vt), \frac{\nu^2}{\lambda^2}y \right) - \phi^{(a', \nu)}\left( \frac{t}{\lambda^2}, \frac{\nu}{\lambda}(x + vt), \frac{\nu^2}{\lambda^2}y \right) \right)\right\|_{H^{s_1, s_2}}^2.  
\end{align*}   
Here 
\begin{align*}
 &\left\|e^{-\frac{i}{2}v \cdot x} \left( \phi^{(a, \nu)}\left( \frac{t}{\lambda^2}, \frac{\nu}{\lambda}(x + vt), \frac{\nu^2}{\lambda^2}y \right) - \phi^{(a', \nu)}\left( \frac{t}{\lambda^2}, \frac{\nu}{\lambda}(x + vt), \frac{\nu^2}{\lambda^2}y \right) \right)\right\|_{H^{s_1, s_2}}^2 \\ 
 &= C\left(\frac{\lambda}{\nu}\right)^3 \int_{\mathbb{R}^2} \left( 1 + \left| \frac{\nu}{\lambda}\xi - \frac{v}{2} \right| \right)^{2s_1} \left( 1 + \left|\frac{\nu^2}{\lambda^2}\eta \right| \right)^{2s_2} \left| \widehat{\phi}^{(a, \nu)} \left( \frac{t}{\lambda^2}, \xi, \eta \right) - \widehat{\phi}^{(a', \nu)} \left( \frac{t}{\lambda^2}, \xi, \eta \right) \right|^2 \, \mathrm{d}\xi \, \mathrm{d}\eta.   
\end{align*} 
Hence from $\nu \leqslant \lambda$ and $s_2 < 0$ we have 
\begin{align}
 &\|u^{(a, \nu, \lambda, v)}(t) - u^{(a', \nu, \lambda, v)}(t)\|_{H^{s_1, s_2}}^2 \notag \\ 
 &\geqslant C \lambda^{-\frac{4}{p-1}} \left( \frac{\lambda}{\nu} \right)^3 \int_{|\xi| \leqslant \frac{\lambda}{4\nu}|v|} |v|^{2s_1} \left\| \mathcal{F}_x[\phi^{(a, \nu)}] \left( \frac{t}{\lambda^2}, \xi, \cdot \right) - \mathcal{F}_x[\phi^{(a', \nu)}] \left( \frac{t}{\lambda^2}, \xi, \cdot \right) \right\|_{H^{s_2}}^2 \, \mathrm{d}\xi \notag \\ 
 &\geqslant C \lambda^{-\frac{4}{p-1}} \left( \frac{\lambda}{\nu} \right)^3 |v|^{2s_1} \Biggl( \left\| \phi^{(a, \nu)}\left( \frac{t}{\lambda^2} \right) - \phi^{(a', \nu)} \left( \frac{t}{\lambda^2} \right) \right\|_{H^{0, s_2}}^2 \notag \\ 
 &\hspace{3em} - \int_{|\xi| \geqslant \frac{\lambda}{4\nu}|v|} \left( \left\| \mathcal{F}_x[\phi^{(a, \nu)}] \left( \frac{t}{\lambda^2}, \xi, \cdot \right) \right\|_{H^{s_2}}^2 + \left\| \mathcal{F}_x[\phi^{(a', \nu)}]\left( \frac{t}{\lambda^2}, \xi, \cdot \right) \right\|_{H^{s_2}}^2 \right)\, \mathrm{d}\xi \Biggr).   
                   \label{solution_difference}
\end{align}  
By Lemma \ref{lemma1}, \eqref{decay_initial} and $s_2 < 0$, we have 
\begin{align}
 \|\phi(t)\|_{H^{k, s_2}} 
 &\leqslant C\|\phi(t)\|_{\mathcal{H}^{k, k}} 
  \leqslant C\|\phi(t) - \phi^{(0)}(t)\|_{\mathcal{H}^{k, k}} + C\|\phi^{(0)}(t)\|_{\mathcal{H}^{k, k}}  \notag \\ 
 &\leqslant C\nu + C(1 + |t|)^k 
  \leqslant C|\log \nu|^C 
                                 \label{upperbound}
\end{align}  
for $|t| \leqslant c|\log \nu|^c$. 
%From $\|\phi(s)\|_{H^{k, s_2}} \leqslant C |\log \nu|^C \ (\forall |s| \leqslant C|\log \nu|^C )$ and 
On the other hand, 
\begin{align}
 \|\phi(t)\|_{H^{k, s_2}} &\geqslant \left( \int_{|\xi| \geqslant \frac{\lambda}{4\nu}|v|} \LR{\xi}^{2k} \|\mathcal{F}_x[\phi](t, \xi, \cdot) \|_{H^{s_2}}^2 \, \mathrm{d}\xi \right)^{\frac{1}{2}} \notag \\ 
 &\geqslant C\left( \frac{\lambda}{\nu}|v| \right)^k \left( \int_{|\xi| \geqslant \frac{\lambda}{4\nu}|v|} \|\mathcal{F}_x[\phi](t, \xi, \cdot) \|_{H^{s_2}}^2 \, \mathrm{d}\xi \right)^{\frac{1}{2}}.  
                                 \label{lowerbound}
\end{align}
\eqref{upperbound} and \eqref{lowerbound} lead  
\begin{align*}
 \int_{|\xi| \geqslant \frac{\lambda}{4\nu}|v|} \left( \left\| \mathcal{F}_x[\phi^{(a, \nu)}]\left( \frac{t}{\lambda^2}, \xi, \cdot \right) \right\|_{H^{s_2}}^2 + \left\| \mathcal{F}_x[\phi^{(a', \nu)}]\left( \frac{t}{\lambda^2}, \xi, \cdot \right) \right\|_{H^{s_2}}^2 \right)\, \mathrm{d}\xi 
 \leqslant C|\log \nu|^C \left( \frac{\lambda}{\nu} |v| \right)^{-2k}  
\end{align*}
when $\left| \frac{t}{\lambda^2} \right| \leqslant c|\log \nu|^c$. 
Then 
\begin{align*}
 &(\text{RHS} \ \text{of} \ \eqref{solution_difference}) \\ 
 &\geqslant C \lambda^{-\frac{4}{p-1}} \left( \frac{\lambda}{\nu} \right)^3 |v|^{2s_1} \left( \left\| \phi^{(a, \nu)}\left( \frac{t}{\lambda^2} \right) - \phi^{(a', \nu)} \left( \frac{t}{\lambda^2} \right) \right\|_{H^{0, s_2}}^2 - C |\log \nu|^C \left( \frac{\lambda}{\nu}|v| \right)^{-2k} \right).
\end{align*}
Hence from $s_1 < 0$ and $|v| \geqslant 1$ we obtain \eqref{solution_lower_bound}.  
%\begin{align*}
% &\|u^{(a, \nu, \lambda, v)}(t) - u^{(a', \nu, \lambda, v)}(t)\|_{H^{s_1, s_2}} \\ 
% &\geqslant C \lambda^{-\frac{2}{p-1}} \left( \frac{\lambda}{\nu} \right)^{\frac{3}{2}} |v|^{s_1} \left( \left\| \phi^{(a, \nu)}\left( \frac{t}{\lambda^2} \right) - \phi^{(a', \nu)}\left( \frac{t}{\lambda^2} \right) \right\|_{L^2_x H^{s_2}_y} - C |\log \nu|^C \left( \frac{\lambda}{\nu} \right)^{-k} |v|^{-s_1 - k} \right)
%\end{align*}

\end{proof}

%%%%%%%%%%%%%%%%%%%%%%%%%%%%%%%%%%%%%%%%%%%%%%%%%%%%%%%%%%%
%%%%%%%%%%%%%%%%%%%%%%% Section 3 %%%%%%%%%%%%%%%%%%%%%%%%%
%%%%%%%%%%%%%%%%%%%%%%%%%%%%%%%%%%%%%%%%%%%%%%%%%%%%%%%%%%%
\section{Norm inflation} 
In this section, we prove Theorem \ref{norm_inflation}. 
We consider either (i) 
$s_1, s_2 \geqslant 0$ except for $s_1 = s_2 = 0$ and $s_1 + 2s_2 < \frac{3}{2}-\frac{2}{p-1}$ 
or 
(ii) 
$s_1, s_2 \leqslant -\frac{1}{2}$.     
In addition we assume that $0 < \lambda \leqslant \nu \ll 1$. 

\begin{proof}[Proof of Theorem \ref{norm_inflation}]
We see from $\phi^{(a, \nu)}(0, x, y) = aw(x, y), 0 \neq w \in \mathcal{S}(\mathbb{R}^2)$ that 
\begin{align*}
 &u^{(a, \nu, \lambda, 0)}(0, x, y) 
 = \lambda^{-\frac{2}{p-1}}\phi^{(a, \nu)}\left( 0, \frac{\nu}{\lambda}x, \frac{\nu^2}{\lambda^2}y \right) 
 = a\lambda^{-\frac{2}{p-1}} w\left( \frac{\nu}{\lambda}x, \frac{\nu^2}{\lambda^2}y \right), \\ 
 &[u^{(a, \nu, \lambda, 0)}(0)]^{\wedge}(\xi, \eta)                                                 
 = a\lambda^{-\frac{2}{p-1}}\left( \frac{\lambda}{\nu} \right)^3 \widehat{w}\left( \frac{\lambda}{\nu}\xi,  \frac{\lambda^2}{\nu^2}\eta \right).   
\end{align*} 
Hence if $\lambda \leqslant \nu$, then 
\begin{align}
 &\|u^{(a, \nu, \lambda, 0)}(0)\|_{H^{s_1, s_2}}^2 
 = a^2 \lambda^{-\frac{4}{p-1}} \left( \frac{\lambda}{\nu} \right)^3 \int_{\mathbb{R}^2} \left( 1 + \left|\frac{\nu}{\lambda}\xi\right|^2 \right)^{s_1} \left( 1 + \left| \frac{\nu^2}{\lambda^2}\eta \right|^2 \right)^{s_2} |\widehat{w}(\xi, \eta)|^2 \, \mathrm{d}\xi \, \mathrm{d}\eta \notag \\ 
 &\leqslant C \lambda^{-\frac{4}{p-1}} \left( \frac{\lambda}{\nu} \right)^{3-4s_2} \left( \int_{|\xi| \geqslant \frac{\lambda}{\nu}} \left( \frac{\nu}{\lambda}|\xi| \right)^{2s_1} \|\mathcal{F}_x[w](\xi, \cdot)\|_{H^{s_2}}^2 \, \mathrm{d}\xi + \int_{|\xi| \leqslant \frac{\lambda}{\nu}} \|\mathcal{F}_x[w](\xi, \cdot)\|_{H^{s_2}}^2 \, \mathrm{d}\xi \right) \notag \\ 
 &= C\lambda^{-\frac{4}{p-1}} \left( \frac{\lambda}{\nu} \right)^{3-2s_1-4s_2} \Biggl( \int_{\mathbb{R}} |\xi|^{2s_1} \|\mathcal{F}_x[w](\xi, \cdot)\|_{H^{s_2}}^2 \, \mathrm{d}\xi \notag \\ 
 &\hspace{7em} - \int_{|\xi| \leqslant \frac{\lambda}{\nu}} \left( |\xi|^{2s_1} - \left( \frac{\lambda}{\nu} \right)^{2s_1}\right) \|\mathcal{F}_x[w](\xi, \cdot)\|_{H^{s_2}}^2 \, \mathrm{d}\xi \Biggr). 
               \label{uanulambda00}
\end{align} 
Thus if $\lambda \leqslant \nu$, then 
\begin{align*}
 \|u^{(a, \nu, \lambda, 0)}(0)\|_{H^{s_1, s_2}} 
 &\leqslant C \lambda^{-\frac{2}{p-1}} \left( \frac{\lambda}{\nu} \right)^{\frac{3}{2}-s_1-2s_2}\left( 1 + O\left(\left( \frac{\lambda}{\nu} \right)^{s_1+\frac{1}{2}} \right) \right) \\ 
 &\leqslant C \lambda^{-\frac{2}{p-1}} \left( \frac{\lambda}{\nu} \right)^{\frac{3}{2}-s_1-2s_2} 
  = C \lambda^{\frac{3}{2}-\frac{2}{p-1}-s_1-2s_2}\nu^{s_1+2s_2-\frac{3}{2}}.  
\end{align*}
Taking $\lambda = \nu^{\sigma}, \sigma := \frac{\frac{3}{2} - s_1 -2s_2 + \delta}{\frac{3}{2} - \frac{2}{p-1} - s_1 - 2s_2} > 1$ and $\delta > 0$ be such that $\nu^{\delta} = \varepsilon$, then we have 
\begin{align*}
 \|u^{(a, \nu, \lambda, 0)}(0)\|_{H^{s_1, s_2}} \leqslant C\varepsilon.  
\end{align*} 
Similarly from $[u^{(a, \nu, \lambda, 0)}(\lambda^2 t)](\xi, \eta) = \lambda^{-\frac{2}{p-1}} \left( \frac{\lambda}{\nu} \right)^3 \widehat{\phi}^{(a, \nu)} \left(t, \frac{\lambda}{\nu}\xi, \frac{\lambda^2}{\nu^2}\eta \right)$, $s_1, s_2 \geqslant 0$ except for $s_1 = s_2 = 0$ and $\lambda \leqslant \nu$, we see 
\begin{align*}
 \|u^{(a, \nu, \lambda, 0)}(\lambda^2 t)\|_{H^{s_1, s_2}}^2 
 &= C\lambda^{-\frac{4}{p-1}} \left( \frac{\lambda}{\nu} \right)^3 \int_{\mathbb{R}^2} \left( 1 + \left|\frac{\nu}{\lambda} \xi \right| \right)^{2s_1} \left( 1 + \left|\frac{\nu^2}{\lambda^2} \eta \right| \right)^{2s_2} |\widehat{\phi}^{(a, \nu)}(t, \xi, \eta)|^2 \, \mathrm{d}\xi \, \mathrm{d}\eta \\ 
 &\geqslant C\lambda^{-\frac{4}{p-1}} \left( \frac{\lambda}{\nu} \right)^3 \int_{|\xi| \geqslant 1, |\eta| \geqslant 1} \left( \frac{\nu}{\lambda} |\xi| \right)^{2s_1} \left( \frac{\nu^2}{\lambda^2} |\eta| \right)^{2s_2} |\widehat{\phi}^{(a, \nu)}(t, \xi, \cdot)|^2 \, \mathrm{d}\xi \, \mathrm{d}\eta \\ 
 &\geqslant C \lambda^{-\frac{4}{p-1}} \left( \frac{\lambda}{\nu} \right)^{3-2s_1-4s_2} (\|\phi^{(a, \nu)}(t)\|_{H^{s_1, s_2}}^2 - \|\phi^{(a, \nu)}(t)\|_{L^2}^2).    
\end{align*}  
We estimate $\|\phi^{(a, \nu)}(t)\|_{H^{s_1, s_2}}$ below. 
If we can show 
\begin{align} 
 \|\phi^{(a, 0)}(t)\|_{H^{s_1, s_2}} \sim t^{s_1+s_2}, \qquad t \gg 1 
                \label{phia0_decay} 
\end{align}  
then we apply Lemma \ref{lemma1} and obtain $\|\phi^{(a, \nu)}(t)\|_{H^{s_1, s_2}} \sim t^{s_1+s_2}$.   
This leads  
\begin{align*}
 \|u^{(a, \nu, \lambda, 0)}(\lambda^2 t)\|_{H^{s_1, s_2}} 
 \geqslant C \lambda^{-\frac{2}{p-1}} \left( \frac{\lambda}{\nu} \right)^{\frac{3}{2} - s_1 - 2s_2}\|\phi^{(a, \nu)}(t)\|_{H^{s_1, s_2}} 
 \sim \varepsilon t^{s_1 + s_2},                                         
\end{align*} 
when $\nu \ll 1$ and $1 \ll t \leqslant c|\log \nu|^c$.  
\eqref{phia0_decay} is confirmed by the same manner as in the proof of Theorem 2 (page 17 \cite{CCT1}). 
Indeed, from 
\begin{align} 
 \phi^{(a, 0)}(t, x, y) = aw(x, y) e^{i\mu a^{p-1}t |w(x, y)|^{p-1}}   
                  \label{phi_a0}
\end{align} 
and 
\begin{align*}
 \partial_x^{j_1} \partial_y^{j_2} \phi^{(a, 0)}(t, x, y) = aw(x, y) t^{j_1+j_2} (i\mu a^{p-1} \nabla_{x, y} |w(x, y)|^{p-1})^{j_1+j_2} e^{i\mu a^{p-1}t|w(x, y)|^{p-1}} + O(t^{j_1+j_2})  
\end{align*}  
for all integers $j_1, j_2 \geqslant 0$ if $p$ is an odd integer, and all $j_1, j_2$ satisfying $0 \leqslant j_1 + j_2 \leqslant p-1$ otherwise, 
we have 
\begin{align*}
 \|\phi^{(a, 0)}(t)\|_{H^{j_1, j_2}} \sim t^{j_1+j_2}. 
\end{align*}   
Under the assumptions of $s_1, s_2, k, p$ in Theorem \ref{norm_inflation}, 
we have \eqref{phia0_decay}. 

Next, we prove the case $s_1, s_2 \leqslant -\frac{1}{2}$. 
From \eqref{uanulambda00} and 
if $\|\mathcal{F}_x[w](\xi, \cdot)\|_{H^{s_2}} \sim O(|\xi|^{\kappa})$ as $\xi \to 0$ for some 
$\kappa > -s-\frac{1}{2}$, then from $\lambda \leqslant \nu$, 
\begin{align*} 
 \int_{|\xi| \leqslant \frac{\lambda}{\nu}} \| \mathcal{F}_x[w](\xi, \cdot)\|_{H^{s_2}} \left( \left( \frac{\lambda}{\nu} \right)^{2s_1} - |\xi|^{2s_1} \right)\, \mathrm{d}\xi 
 \leqslant C\left( \frac{\lambda}{\nu} \right)^{2\kappa + 2s_1 + 1} 
 < \infty. 
\end{align*}
This leads $\|u^{(a, \nu, \lambda, 0)}(0)\|_{H^{s_1, s_2}} \leqslant C\varepsilon$. 
We will show that for any $\varepsilon > 0$, 
it holds $\|u^{(a, \nu, \lambda, 0)}(\lambda^2)\|_{H^{s_1, s_2}} > \frac{1}{\varepsilon}$ 
for sufficiently small $0 < \lambda \leqslant \nu$.  
We check the case $s_1, s_2 < -\frac{1}{2}$ at first. 
%Then $s_1 + 2s_2 < -\frac{3}{2}$. 
\begin{align*}
 \left|\int_{\mathbb{R}^2}\phi^{(a, 0)}(1, x, y)\, \mathrm{d}x \, \mathrm{d}y\right| \geqslant C 
\end{align*}
is equivalent to $|\widehat{\phi}^{(a, 0)}(1, 0, 0)| \geqslant C$. 
Here $\phi^{(a, 0)}(1)$ is rapidly decreasing, hence by continuity we see for $|\xi|, |\eta| \leqslant c$ with $0 < c \ll 1$ that 
\begin{align}
 |\widehat{\phi}^{(a, 0)}(1, \xi, \eta)| \geqslant C. 
                \label{phia01}
\end{align}  
By $|\xi|, |\eta| \leqslant c \ll 1$, the Cauchy-Schwarz inequality and Lemma \ref{lemma1}, we obtain  
\begin{align}
 |\widehat{\phi}^{(a, \nu)}(1, \xi, \eta) - \widehat{\phi}^{(a, 0)}(1, \xi, \eta)| 
 &\leqslant \int_{|\xi|, |\eta| \leqslant c \ll 1} |\phi^{(a, \nu)}(1, \xi, \eta) - \phi^{(a, 0)}(1, \xi, \eta)| \, \mathrm{d}\xi \, \mathrm{d}\eta \notag \\ 
 &\leqslant C\| \phi^{(a, \nu)}(1) - \phi^{(a, 0)}(1)\|_{\mathcal{H}^{k, k}}  
  \leqslant C\nu. 
                   \label{difference_phia01}
\end{align}
Collecting \eqref{phia01} and \eqref{difference_phia01}, we have   
\begin{align}
 |\widehat{\phi}^{(a, \nu)}(1, \xi, \eta)| \geqslant C
                      \label{phianu1}
\end{align} 
for $|\xi|, |\eta| \leqslant c$ and sufficiently small $\nu > 0$. 
From \eqref{u}, we compute 
\begin{align*}
 \widehat{u}^{(a, \nu, \lambda, 0)}(\lambda^2, \xi, \eta) = \lambda^{-\frac{2}{p-1}} \left( \frac{\lambda}{\nu} \right)^3 \widehat{\phi}^{(a, \nu)}\left( 1, \frac{\lambda}{\nu}\xi, \frac{\lambda^2}{\nu^2}\eta \right). 
\end{align*} 
If $|\xi| \leqslant c\frac{\nu}{\lambda}, |\eta| \leqslant c\frac{\nu^2}{\lambda^2}$, then by \eqref{phianu1}, we have 
$|\widehat{u}^{(a, \nu, \lambda, 0)}(\lambda^2, \xi, \eta)| 
 \geqslant C\lambda^{-\frac{2}{p-1}}\left( \frac{\lambda}{\nu} \right)^3$. 
From this and $s_1, s_2 < -\frac{1}{2}$ lead  
\begin{align}
 \|u^{(a, \nu, \lambda, 0)}(\lambda^2)\|_{H^{s_1, s_2}} 
 &\geqslant C\lambda^{-\frac{2}{p-1}} \left( \frac{\lambda}{\nu} \right)^3 \left( \int_{|\xi| \leqslant c\frac{\nu}{\lambda}, |\nu| \leqslant c\frac{\nu^2}{\lambda^2}} (1+|\xi|)^{2s_1} (1+|\eta|)^{2s_2} \, \mathrm{d}\xi \, \mathrm{d}\eta \right)^{\frac{1}{2}} 
                         \label{ulambda^2} \\ 
 &%\geqslant C \lambda^{-\frac{2}{p-1}} \left( \frac{\lambda}{\nu} \right)^3 
 = C\lambda^{-\frac{2}{p-1}} \left( \frac{\lambda}{\nu}\right)^{\frac{3}{2}-s_1-2s_2} \left( \frac{\lambda}{\nu} \right)^{\frac{3}{2}+s_1+2s_2}  
 = C\varepsilon \left( \frac{\lambda}{\nu} \right)^{\frac{3}{2}+s_1+2s_2}.  \notag 
\end{align}
If $\lambda \leqslant \nu$, then 
$\left( \frac{\lambda}{\nu} \right)^{\frac{3}{2}+s_1+2s_2} \to \infty$ as $\nu \to 0$ 
since $s_1 + 2s_2 < -\frac{3}{2}$. 
Thus the claim follows by taking $\nu$ sufficiently small depending on $\varepsilon$. 
We consider the case $s_1 = s_2 = -\frac{1}{2}$ secondly. 
From \eqref{ulambda^2}, $\lambda \leqslant \nu$ and 
\begin{align*}
 \int_{|\xi| \leqslant c\frac{\nu}{\lambda}, |\nu| \leqslant c\frac{\nu^2}{\lambda^2}} (1+|\xi|)^{-1} (1+|\eta|)^{-1} \, \mathrm{d}\xi \, \mathrm{d}\eta 
 \geqslant C\log \left( \frac{\nu}{\lambda} \right), 
\end{align*} 
we obtain 
\begin{align*}
 \|u^{(a, \nu, \lambda, 0)}(\lambda^2)\|_{H^{s_1, s_2}} 
 \geqslant C\lambda^{-\frac{2}{p-1}}\left( \frac{\lambda}{\nu} \right)^3 \log \left( \frac{\nu}{\lambda} \right)        
 = C\varepsilon \log \left( \frac{\nu}{\lambda} \right). 
\end{align*}   
Thus the desired result follows. 
It remains to show the claim when $s_1 = -\frac{1}{2}, s_2 < -\frac{1}{2}$ or $s_1 < -\frac{1}{2}, s_2 = -\frac{1}{2}$. 
These cases are handled in a similar way, hence we omit the detail. 

\end{proof}

%As a corollary, we have the following weaker result, namely the solution map is not uniformly continuous in $H^{s_1, s_2}$ where $s_1, s_2$ satisfy Theorem \ref{norm_inflation}. 
%\begin{cor} 
% Let $p > 1$ be an odd integer and $\mu = \pm 1$. 
% For any $s_1, s_2 > 0$ and $s_1 + 2s_2 < \frac{3}{2} - \frac{2}{p-1}$. 
% Then for any $0 < \delta, \varepsilon < 1$ and $t > 0$ 
% there exist solutions of \eqref{NLS} with initial data $u_1(0), u_2(0) \in \mathcal{S}(\mathbb{R}^2)$ such that 
% \begin{align}
%  &\|u_1(0)\|_{H^{s_1, s_2}}, \|u_2(0)\|_{H^{s_1, s_2}} \leqslant C\varepsilon, \label{estimate_initialdata} \\ 
%  &\|u_1(0) - u_2(0)\|_{H^{s_1, s_2}} \leqslant C\delta, \label{difference_initialdata} \\ 
%  &\|u_1(t) - u_2(t)\|_{H^{s_1, s_2}} \geqslant C\varepsilon. \label{difference_solution}
% \end{align}  
%If $p > 1$ is not an odd integer, then the same conclusion holds provided that there exists an integer $k > \frac{3}{2}$ such that $p \geqslant k + 1$ and $s_1, s_2 > 0, s_1 + 2s_2 < \frac{3}{2} - \frac{2}{p-1}$. 
%\end{cor} 

%%%%%%%%%%%%%%%%%%%%%%%%%%%%%%%%%%%%%%%%%%%%%%%%%%%%%%
%%%%%%%%%%%%%%%%%%%% Section 4 %%%%%%%%%%%%%%%%%%%%%%%
%%%%%%%%%%%%%%%%%%%%%%%%%%%%%%%%%%%%%%%%%%%%%%%%%%%%%%
\section{Decoherence} 
In this section, we prove Theorem \ref{Decoherence}. 

\begin{proof}[Proof of Theorem \ref{Decoherence}] 
Firstly, we consider the case $s_1, s_2 < 0$. 
Set $\lambda^{-\frac{2}{p-1}} \left( \frac{\lambda}{\nu} \right)^{\frac{3}{2}} |v|^{s_1} = \varepsilon$, $\lambda = \nu^{\sigma}$ such that $0 < \varepsilon, \sigma \ll 1$ and $\sigma$ will be fixed later. 
This condition is equivalent to $|v| = \nu^{\frac{1}{s_1}\left( \frac{3}{2}(1-\sigma) + \frac{2\sigma}{p-1} \right)} \cdot \varepsilon^{\frac{1}{s_1}}$, namely $|v|$ grows negative power of $\nu$. 
From \eqref{est_initialdata} and \eqref{est_difference_initialdata} we obtain 
\begin{align*}
 &\|u^{(a, \nu, \lambda, v)}(0)\|_{H^{s_1, s_2}} + \|u^{(a', \nu, \lambda, v)}(0)\|_{H^{s_1, s_2}} \leqslant C\varepsilon, \\ 
 &\|u^{(a, \nu, \lambda, v)}(0) - u^{(a', \nu, \lambda, v)}(0)\|_{H^{s_1, s_2}} \leqslant C\varepsilon |a-a'|.  
\end{align*} 
Taking $\delta = \varepsilon |a - a'|$ then we obtain \eqref{estimate_initialdata} and \eqref{difference_initialdata}. 
There exists some $T = T(a, a') > 0$ such that  
\begin{align} 
 \|\phi^{(a, 0)}(T) - \phi^{(a', 0)}(T)\|_{H^{0, s_2}} 
 = \|aw e^{i\mu T a^{p-1} |w|^{p-1}} - a'w e^{i\mu T a'^{p-1}|w|^{p-1}}\|_{H^{0, s_2}}  
 \geqslant C.  
                  \label{nu00lower}
\end{align}
Here $C > 0$ is independent of $a, a'$ and fix this $T$. 
By Lemma \ref{lemma1}, for all $0 < T \leqslant c|\log \nu|^c$ 
\begin{align}
 \|\phi^{(a, \nu)}(T) - \phi^{(a, 0)}(T)\|_{H^{k, k}(\mathbb{R}^2)} \leqslant C\nu,  
                  \label{nu0bound}
\end{align} 
where $\phi^{(a, 0)}(t)$ is defined in \eqref{phi_a0}.   
From \eqref{nu00lower} and \eqref{nu0bound}, we have    
\begin{align*}
 \|\phi^{(a, \nu)}(T) - \phi^{(a', \nu)}(T)\|_{H^{0, s_2}} 
 &\geqslant \|\phi^{(a, 0)}(T) - \phi^{(a', 0)}(T)\|_{H^{0, s_2}} - \|\phi^{(a, \nu)}(T) - \phi^{(a, 0)}(T)\|_{H^{0, s_2}} \\ 
 &\hspace{3em} - \|\phi^{(a', 0)}(T) - \phi^{(a', \nu)}(T)\|_{H^{0, s_2}} \\ 
 &\geqslant C - 2C\nu 
  \geqslant C    
\end{align*} 
provided $s_2 < 0$ and $\nu > 0$ is sufficiently small. 
From \eqref{solution_lower_bound}, if $0 < T \leqslant c|\log \nu|^c$, then we have 
\begin{align}
 &\|u^{(a, \nu, \lambda, v)}(\lambda^2 T) - u^{(a', \nu, \lambda, v)}(\lambda^2 T)\|_{H^{s_1, s_2}} \notag \\ 
 &\geqslant C\varepsilon \left(\|\phi^{(a, \nu)}(T) - \phi^{(a', \nu)}(T)\|_{H^{0, s_2}} - C |\log \nu|^C \left( \frac{\lambda}{\nu} \right)^{-k}|v|^{-s_1-k}\right) \notag \\ 
 &\geqslant C\varepsilon - C\varepsilon \left( \frac{\lambda}{\nu} \right)^{-k} |v|^{-s_1-k}|\log \nu|^C.   
                    \label{decoherence}
\end{align} 
We set $\lambda = \nu^{\sigma}$ with sufficiently small $\sigma > 0$ such that  
\begin{align*}
 \left( \frac{\lambda}{\nu} \right)^{-k} |v|^{-s_1-k}|\log \nu|^C 
 = \nu^{k-\frac{3(s_1+k)}{2s_1}+\left( \frac{s_1+k}{s_1}\left( \frac{3}{2} - \frac{2}{p-1} \right) -k  \right) \sigma} |\log \nu|^C \, \varepsilon^{-\frac{s_1+k}{s_1}} 
 \to 0   
\end{align*} 
as $\nu \to 0$. Indeed, if $s_1 < 0$ and $k > \frac{3}{2}$ we see 
$k - \frac{3(s_1+k)}{2s_1} > -\frac{9}{4}s_1 > 0$. 
Therefore 
\begin{align*}
 (\text{RHS}\ \text{of}\ \eqref{decoherence}) \geqslant C\varepsilon. 
\end{align*}  
As $\nu \to 0$, we have $\lambda = \nu^{\sigma} \to 0$, hence $\lambda^2 T \to 0$.  
Therefore we conclude \eqref{difference_solution}.  
Secondly we check the case $s_1 = s_2 = 0$. 
We see 
\begin{align*}
 \|\phi^{(a, 0)}(t) - \phi^{(a', 0)}(t)\|_{L^2} \geqslant C > 0 
\end{align*} 
when $t \geqslant C|a - a'|^{-1}$. 
If $\nu > 0$ is sufficiently small and $C|a-a'|^{-1} \leqslant t \leqslant c|\log \nu|^c$, then \eqref{est_small_dispersion} gives 
\begin{align*}
 \|\phi^{(a, \nu)}(t) - \phi^{(a', \nu)}(t)\|_{L^2} 
 &\geqslant \|\phi^{(a, 0)}(t) - \phi^{(a', 0)}(t)\|_{L^2} - \|\phi^{(a, \nu)}(t) - \phi^{(a, 0)}(t)\|_{L^2} \\ 
 &\hspace{3em} - \|\phi^{(a', 0)}(t) - \phi^{(a', \nu)}(t)\|_{L^2} \\ 
 &\geqslant C - 2C\nu 
  \geqslant C.      
\end{align*} 
We see 
\begin{align*}
 \|u^{(a, \nu, \lambda, 0)}(\lambda^2 t) - u^{(a', \nu, \lambda, 0)}(\lambda^2 t) \|_{L^2} 
 = \lambda^{-\frac{2}{p-1}} \left( \frac{\lambda}{\nu} \right)^{\frac{3}{2}} \|\phi^{(a, \nu)}(t) - \phi^{(a', \nu)}(t)\|_{L^2} 
 \geqslant C\varepsilon. 
\end{align*}
Of course it holds that 
\begin{align*}
 \|u^{(a, \nu, \lambda, 0)}(0)\|_{L^2} \leqslant C\varepsilon,\qquad 
 \|u^{(a, \nu \lambda, 0)}(0) - u^{(a', \nu, \lambda, 0)}(0)\|_{L^2} \leqslant C\varepsilon |a - a'|. 
\end{align*}
Therefore the claim follows for the case $s_1 = s_2 = 0$. 
The cases $s_1 = 0$ and $s_2 < 0$ or $s_1 < 0$ and $s_2 = 0$ would be similar, hence we omit the proof.  
  
\end{proof}

%%%%%%%%%%%%%%%%%%%%%%%%%%%%%%%%%%%%%%%%%%%%%%%%%%%%%%%%
%%%%%%%%%%%%%%%%%%%%%% Section 5 %%%%%%%%%%%%%%%%%%%%%%%
%%%%%%%%%%%%%%%%%%%%%%%%%%%%%%%%%%%%%%%%%%%%%%%%%%%%%%%%
\section{Decoherence at the critical regularity for focusing case}
In this section, we consider \eqref{NLS} with $\mu = -1$ (focusing case): 
%the following focusing half wave Schr\"odinger equation ($\mu = -1$): 
\begin{align}
 i\partial_t u + \partial_x^2 u - |D_y|u = -|u|^{p-1}u. 
               \label{focusing_hws}
\end{align} 
For $1 < p < 5$, there exists the standing wave for \eqref{focusing_hws}, namely 
$u(t, x, y) = e^{i\beta t} Q_{\beta}(x, y), \beta > 0$. 
Here $Q_{\beta}$ satisfies the following equation (see (1.5) \cite{BIK}): 
\begin{align*} 
 -\partial_x^2 Q + |D_y|Q + \beta Q - |Q|^{p-1}Q = 0. 
\end{align*}   
%Also there exists the traveling wave for \eqref{focusing_hws} when $1 < p < 5$ (see (1.8) \cite{BIK}). 
%That is, $u(t, x, y) = e^{i\beta t} Q_{\beta, \gamma}(x, y-\gamma t), \beta > 0, \gamma \in \mathbb{R}$. 
%$Q_{\beta, \gamma}$ satisfies 
%\begin{align*}
% -\partial_x^2 Q + |D_y|Q + i\gamma \partial_y Q + \beta Q - |Q|^{p-1} Q = 0. 
%\end{align*} 
%We note that 
%\begin{align}
% Q_{\beta, \gamma}(x, y) = \beta^{\frac{1}{p-1}} Q_{1, \gamma}(\sqrt{\beta}x, \beta y). 
%             \label{scale}
%\end{align} 
We note that 
\begin{align}
 Q_{\beta}(x, y) = \beta^{\frac{1}{p-1}} Q_1(\sqrt{\beta}x, \beta y). 
              \label{scale} 
\end{align}

\begin{proof}[Proof of Theorem \ref{decoherence_critical}] 
We set two solutions $u_1, u_2$ of \eqref{focusing_hws} as  
\EQ{ \label{standing_wave}
 u_j(t, x, y) = e^{i\beta_j t} Q_{\beta_j}(x, y), 
 %u_2(t, x, y) &= e^{i\beta_2 t} Q_{\beta_2}(x, y),  
} 
where $\beta_j > 0, j= 1, 2$.  
%From \eqref{scale}, change of variables and $s_1 + 2s_2 = \frac{3}{2} - \frac{2}{p-1}$, we see for $i = 1, 2$ 
%\begin{align*}
% \|u_i(0)\|_{\dot{H}^{s_1, s_2}}
% &= \|Q_{\beta_i}\|_{\dot{H}^{s_1, s_2}}
%  = \beta_i^{\frac{1}{p-1}} \|Q_1(\sqrt{\beta_i} \, \cdot, \beta_i \, \cdot)\|_{\dot{H}^{s_1, s_2}} \\ 
% &= \beta_i^{\frac{1}{p-1} - \frac{3}{2}} \left( \int_{\mathbb{R}^2} |\xi|^{2s_1} |\eta|^{2s_2} \left| \widehat{Q}_1 \left( \frac{\xi}{\sqrt{\beta_i}}, \frac{\eta}{\beta_i} \right) \right|^2 \, \mathrm{d}\xi \, \mathrm{d}\eta \right)^{\frac{1}{2}} \\ 
% &= \beta_i^{\frac{1}{p-1} - \frac{3}{2} + \frac{s_1}{2} + s_2 + \frac{3}{4}} \left( \int_{\mathbb{R}^2} |\xi|^{2s_1} |\eta|^{2s_2} \left| \widehat{Q}_1 (\xi, \eta)  \right|^2 \, \mathrm{d}\xi \, \mathrm{d}\eta \right)^{\frac{1}{2}} 
%  = \|Q_1\|_{\dot{H}^{s_1, s_2}}. 
%\end{align*} 
%We take $\varepsilon = \|Q_1\|_{H^{s_1, s_2}}$ then 
%\eqref{estimate_initialdata} follows. 
From \eqref{scale}, \eqref{standing_wave}, change of variables and $s_1 + 2s_2 = \frac{3}{2} - \frac{2}{p-1}$ we see 
\begin{align*}
 &\|u_1(0) - u_2(0)\|_{\dot{H}^{s_1, s_2}}^2 
  = \|Q_{\beta_1} - Q_{\beta_2}\|_{\dot{H}^{s_1, s_2}}^2 \\ 
 &= \int_{\mathbb{R}^2} |\xi|^{2s_1} |\eta|^{2s_2} \left| \beta_1^{\frac{1}{p-1} - \frac{3}{2}} \widehat{Q}_1 \left( \frac{\xi}{\sqrt{\beta_1}}, \frac{\eta}{\beta_1}  \right) - \beta_2^{\frac{1}{p-1} - \frac{3}{2}} \widehat{Q}_1 \left( \frac{\xi}{\sqrt{\beta_2}}, \frac{\eta}{\beta_2} \right) \right|^2 \, \mathrm{d}\xi \, \mathrm{d}\eta \\ 
 &= \int_{\mathbb{R}^2} |\xi|^{2s_1} |\eta|^{2s_2} |\widehat{Q}_1(\xi, \eta)|^2 \, \mathrm{d}\xi \, \mathrm{d}\eta + \left( \frac{\beta_2}{\beta_1} \right)^{\frac{2}{p-1} - 3} \int_{\mathbb{R}^2} |\xi|^{2s_1} |\eta|^{2s_2} \left| \widehat{Q}_1\left( \sqrt{\frac{\beta_1}{\beta_2}}\, \xi, \frac{\beta_1}{\beta_2}\, \eta \right) \right|^2 \, \mathrm{d}\xi \, \mathrm{d}\eta \\ 
 &\qquad - 2 \left( \frac{\beta_2}{\beta_1} \right)^{\frac{1}{p-1} - \frac{3}{2}} \int_{\mathbb{R}^2} |\xi|^{2s_1} |\eta|^{2s_2} \, \widehat{Q}_1(\xi, \eta) \, \overline{\widehat{Q}}_1\left( \sqrt{\frac{\beta_1}{\beta_2}}\, \xi, \frac{\beta_1}{\beta_2}\, \eta \right) \, \mathrm{d}\xi \, \mathrm{d}\eta \\ 
 &\to 0  
\end{align*} 
as $\frac{\beta_2}{\beta_1} \to 1$. 
Hence \eqref{difference_initialdata} follows. 
Finally, we check \eqref{difference_solution}. 
\begin{align*}
 \|u_1(t) - u_2(t)\|_{\dot{H}^{s_1, s_2}}^2 
 = \|u_1(t)\|_{\dot{H}^{s_1, s_2}}^2 + \|u_2(t)\|_{\dot{H}^{s_1, s_2}}^2 - 2\LR{u_1(t), u_2(t)}_{s_1, s_2}.   
\end{align*} 
Similar to the above argument, we see 
$\|u_i(t)\|_{\dot{H}^{s_1, s_2}} = \|Q_1\|_{\dot{H}^{s_1, s_2}}$ for $i = 1, 2$. 
\begin{align*}
 &\LR{u_1(t), u_2(t)}_{s_1, s_2} 
  = \int_{\mathbb{R}^2} |\xi|^{2s_1} |\eta|^{2s_2} e^{i(\beta_1 - \beta_2)t}\,  \widehat{Q}_{\beta_1}(\xi, \eta)\, \overline{\widehat{Q}}_{\beta_2}(\xi, \eta) \, \mathrm{d}\xi \, \mathrm{d}\eta \\ 
 &= e^{i(\beta_1 - \beta_2)t} (\beta_1 \beta_2)^{\frac{1}{p-1} - \frac{3}{2}} \int_{\mathbb{R}^2} |\xi|^{2s_1} |\eta|^{2s_2}\, \widehat{Q}_1\left( \frac{\xi}{\sqrt{\beta_1}}, \frac{\eta}{\beta_1} \right)\, \overline{\widehat{Q}}_1\left( \frac{\xi}{\sqrt{\beta_2}}, \frac{\eta}{\beta_2} \right)\, \mathrm{d}\xi \, \mathrm{d}\eta \\ 
 &= e^{i(\beta_1 - \beta_2)t} (\beta_1 \beta_2)^{\frac{1}{p-1} - \frac{3}{2}} \beta_1^{s_1 + 2s_2 + \frac{3}{2}} \int_{\mathbb{R}^2} |\xi|^{2s_1} |\eta|^{2s_2} \,  \widehat{Q}_1(\xi, \eta) \, \overline{\widehat{Q}}_1 \left( \sqrt{\frac{\beta_1}{\beta_2}}\, \xi, \frac{\beta_1}{\beta_2}\, \eta \right)\, \mathrm{d}\xi \, \mathrm{d}\eta.  
\end{align*} 
If we take $\beta_1 = n + 1, \beta_2 = n \in \mathbb{N}, t = (2n + 1)\pi$, then we obtain 
\begin{align*}
 \LR{u_1(t), u_2(t)}_{s_1, s_2} \to -\|Q_1\|_{\dot{H}^{s_1, s_2}}^2,\qquad  n \to \infty  
\end{align*} 
since $s_1 + 2s_2 = \frac{3}{2} - \frac{2}{p-1}$. 
Therefore 
\begin{align*}
 \|u_1(t) - u_2(t)\|_{\dot{H}^{s_1, s_2}}^2 \to 4\|Q_1\|_{\dot{H}^{s_1, s_2}}^2, \qquad n \to \infty. 
\end{align*} 
Hence \eqref{difference_solution} follows.  

\end{proof}

\section*{Acknowledgement}
The author is grateful to Prof. Yoshio Tsutsumi for carefully proofreading the manuscript and also would like to thank Prof. Hiroaki Kikuchi and Dr. Masayuki Hayashi JSPS Research Fellow for many useful discussions and for teaching me the property of the standing wave and the traveling wave solutions of this equation.  
The author is supported by JSPS KAKENHI Grant Number 820200500051.

\end{document}